\documentclass[]{article}

\usepackage[]{graphicx}
\usepackage{amsmath,amssymb,amsthm} 
\usepackage{mathrsfs,mathtools} 
\usepackage{tabularx,booktabs} 
\usepackage{dsfont}

\usepackage{hyperref}
\hypersetup{colorlinks,urlcolor=blue}
\hypersetup{
    colorlinks,
    linkcolor={blue},
    citecolor={blue},
    urlcolor={blue}
}

\usepackage{scalerel}
\def\depthgrowth{0pt}
\def\heightgrowth{1pt}
\newsavebox\zbox
\newcommand\zsqrt[1]{%
  \ignoremathstyle
  \savebox\zbox{$#1\rule{0pt}{.7\baselineskip}$}%
  \stretchrel*{\sqrt{\phantom{#1}\kern0.5pt}}%
              {\rule[-\dimexpr\dp\zbox+\depthgrowth]{0pt}{%
                \dimexpr\ht\zbox+\dp\zbox+\depthgrowth+\heightgrowth}}%
  \kern-\wd\zbox\textstyle#1%
}

\newtheorem{lemma}{Lemma}
\newtheorem{theorem}{Theorem}

\newtheorem{definition}{Definition}
\newtheorem{corollary}{Corollary}
\newtheorem{relation}{Relation}

\DeclareMathOperator*{\arginf}{arg\,inf}

\newcommand\pFq[5]{{_{#1}F_{#2}}\left({#3 \atop #4};#5\right)}

\author{Aaron Hendrickson\\ ahendr16@jh.edu\\ \vspace*{-2pt}\\ Claude F. Leibovici\\ cfl-consultant@club-internet.fr}
\title{Exact and approximate solutions to the minimum of $1+x+\cdots+x^{2n}$}

\begin{document}
\maketitle

\begin{abstract}
The polynomial $f_{2n}(x)=1+x+\cdots+x^{2n}$ and its minimizer on the real line $x_{2n}=\arginf f_{2n}(x)$ for $n\in\Bbb N$ are studied. Results show that $x_{2n}$ exists, is unique, corresponds to $\partial_x f_{2n}(x)=0$, and resides on the interval $[-1,-1/2]$ for all $n$. It is further shown that $\inf f_{2n}(x)=(1+2n)/(1+2n(1-x_{2n}))$ and $\inf f_{2n}(x)\in[1/2,3/4]$ for all $n$ with an exact solution for $x_{2n}$ given in the form of a finite sum of hypergeometric functions of unity argument. Perturbation theory is applied to generate rapidly converging and asymptotically exact approximations to $x_{2n}$. Numerical studies are carried out to show how many terms of the perturbation expansion for $x_{2n}$ are needed to obtain suitably accurate approximations to the exact value.
\end{abstract}

\begin{table}[b]\footnotesize\hrule\vspace{1mm}
\end{table}


\section{Introduction}
\label{sec:introduction}

The inspiration for this work came from a question posted by Wang \cite{wang_4060608} on the \emph{Mathematics Stack Exchange} discussion board March 13, 2021, which sought a solution to the minimum of the polynomial $1+x+\cdots+x^{2n}$ for $n\in\Bbb N$. In the question is was noted that the minimum appeared to correspond to a vanishing derivative and thus could be found by solving for the real roots of $\partial_x(1+x+\cdots x^{2n})$. When $n=1,2$ these roots are algebraic with their exact forms being recovered using the standard formulae for linear and cubic equations. However for $n\geq 3$, the work of Abel and Galois shows no general algebraic solution exists; hence, motivating the need for more powerful methods \cite{MR1577598}. Given the broad and pervasive applications of geometric sums in the literature, further study of this polynomial and its minimum is a worthwhile venture.


\section{Preliminaries}
\label{sec:prelims}

Throughout this work we define $\Bbb N=\{1,2,\dots\}$, $\Bbb N_0=\Bbb N\cup\{0\}$, and $\Bbb E=\{2,4,\dots\}$ to be the sets of positive integers, nonnegative integers, and positive even integers, respectively. For the sake of brevity we shall denote $m=2n$ so that the the polynomial of interest and its minimizer becomes
\[
f_m(x)\coloneqq 1+x+\cdots+x^m,\quad m\in\Bbb E
\]
and
\[
x_m\coloneqq\arginf_{x\in\Bbb R}f_m(x).
\]
The following definitions and relations will also be used.

\begin{definition}[Gamma function]
\label{def:poch_sym}
\[
\Gamma(s) \coloneqq \int_0^\infty t^{s-1}e^{-t}\,\mathrm dt,\quad \Re s>0
\]
\end{definition}


\begin{definition}[Factorial power (falling factorial)]
\label{def:FactorialPower}
\[
(s)^{(n)}\coloneqq \frac{\Gamma(s+1)}{\Gamma(s-n+1)}
\]
\end{definition}


\begin{definition}[Pochhammer symbol (rising factorial)]
\label{def:pochhammer_symbol}
\[
(s)_n \coloneqq \frac{\Gamma(s+n)}{\Gamma(s)}
\]
\end{definition}


\begin{definition}[Generalized hypergeometric series]
\label{def:pFq_function}
\[
\pFq{p}{q}{a_1,\dots,a_p}{b_1,\dots,b_q}{z}\coloneqq\sum_{k=0}^\infty\frac{(a_1)_k\cdots (a_p)_k}{(b_1)_k\cdots (b_q)_k}\frac{z^k}{k!}
\]
\end{definition}


\begin{definition}[$k$-gamma function and Pochhammer $k$-symbol \cite{diaz_2007}]
\label{def:k_gamma_function}
The $k$-gamma function and Pochhammer $k$-symbol are given by
\[
\Gamma_k(x)\coloneqq k^{\frac{x}{k}-1}\Gamma\left(\frac{x}{k}\right)
\]
and
\[
(x)_{n,k}\coloneqq \frac{\Gamma_k(x+nk)}{\Gamma_k(x)},
\]
respectively.
\end{definition}


\begin{relation}
If $n\in\Bbb N_0$ then $(\alpha)_{n,k}=k^n(\alpha/k)_n$ {\normalfont\cite[Prop.~3.1]{JIMS252}}.
\end{relation}


With these definitions at hand we are ready to begin studying the properties of $f_m$ and $x_m$.


\section{Properties of $f_m$ and $x_m$}
\label{sec:fm_xm_properties}

Our first goal is to establish the existence and uniqueness of $x_m$. To accomplish this it will be helpful to use the closed-form for geometric sums and write $f_m$ in the form
\begin{equation}
\label{eq:fn_closed_form}
f_m(x)=\frac{1-x^{m+1}}{1-x}.
\end{equation}

\begin{lemma}
\label{lem:fn_is_strictly_convex}
The polynomial $f_m(x)$ is strictly convex on $x\in\Bbb R$ for all $m\in\Bbb E$.
\end{lemma}

\begin{proof}
To establish strict convexity it is sufficient to show $f_m^{\prime\prime}(x)>0$ everywhere on $x\in\Bbb R$. It is trivial to show $f_m^{\prime\prime}(x)>0$ holds for $x\geq 0$ so all that is left to do is to consider the complementary case $x<0$. Equating the second derivative with zero we find $f_m^{\prime\prime}(x)=0\iff h_m(x)=0$, where
\[
h_m(x)=(m-1)mx^{m+1}-2(m^2-1)x^m+m(m+1)x^{m-1}-2.
\]
The signs of the coefficients of $h_m(-x)$ in order of descending variable exponent yields the sequence $(-1,-1,-1,-1)$, which are all negative. It follows from Descartes' rule of signs that $f_m^{\prime\prime}(x)$ has zero roots on the interval $x\in(-\infty,0)$. But, $f_m^{\prime\prime}(-1)=\frac{1}{2}m^2>0$; thus, we conclude $f_m^{\prime\prime}(x)>0$ also holds for all $x<0$. The proof is now complete.
\end{proof}

\begin{theorem}
\label{thm:fn_has_a_unique_minimizer}
The minimizer $x_m$ exists, is unique, and resides on the interval $[-1,-1/2]$ for all $m\in\Bbb E$, .
\end{theorem}

\begin{proof}
We begin by establishing that $f_m^\prime(x)$ has exactly one real root. It is immediately obvious that $f_m^\prime(x)>0$ for all $x\geq 0$. Now assuming $x<0$, we deduce $f_m^\prime(x)=0\iff g_m(x)=0$, where $g_m(x)=mx^{m+1}-(m+1)x^m+1$. The signs of the coefficients of $g_m(-x)$ in order of descending variable exponent gives the sequence $(-1,-1,+1)$; revealing a single variation in sign. Again appealing to Descartes' rule of signs we conclude $f_m^\prime(x)$ must have exactly one real root on the interval $x\in(-\infty,0)$. However, $f_m^\prime(-1)=-\frac{1}{2}m$ and $f_m^\prime(-1/2)=\frac{1}{9}2^{1-m}(2^{m+1}-3m-2)\geq 0$ with the latter inequality following from induction on $m\in\Bbb E$. Consequently, $f_m^\prime(x)$ has a single root on the real line contained in the interval $[-1,-1/2]$ for all $m\in\Bbb E$. Furthermore, the strict convexity of $f_m$ proven in Lemma \ref{lem:fn_is_strictly_convex} implies that the solution to $f_m^\prime(x)=0$ also corresponds to the unique global minimum of $f_m$, which completes the proof.
\end{proof}

With the existence and uniqueness of $x_m$ established, we turn to finding a simple formula for the minimum of $f_m$ as a function of $x_m$.

\begin{lemma}
\label{lem:fmin_explicit_form}
Let $x_m\in[-1,-1/2]$ denote the unique minimizer of $f_m$ such that $f_m(x_m)=\inf_{x\in\Bbb R}f_m(x)$. Then,
\[
f_m(x_m)=\frac{1+m}{1+m(1-x_m)}
\]
and $f_m(x_m)\in[1/2,3/4]$ for all $m\in\Bbb E$ with $f_2(x_2)=3/4$ and $\lim_{m\to\infty}f_m(x_m)=1/2$.
\end{lemma}

\begin{proof}
From Theorem \ref{thm:fn_has_a_unique_minimizer} we know that $x_m$ satisfies $mx_m^{m+1}-(m+1)x_m^m+1=0$, which can be rewritten as $x_m^{m+1}=x_m/(1+m(1-x_m))$. Substituting this expression for $x_m^{m+1}$ into (\ref{eq:fn_closed_form}) yields the desired form for $f_m(x_m)$. The bounds on $f_m(x_m)$ are then found by minimizing and maximizing $f(m,x)=(1+m)/(1+m(1-x))$ on $(m,x)\in\Bbb E\times[-1,-1/2]$. We find $\inf f(m,x)=\lim_{m\to\infty}f(m,-1)=1/2$ and $\sup f(m,x)=f(2,-1/2)=3/4$, which are equivalent to $\lim_{m\to\infty}f_m(x_m)$ and $f_2(x_2)$, respectively. The proof is now complete.
\end{proof}


\section{Explicit expression for $x_m$}
\label{sec:xm_exact_form}

In the previous section we showed that the minimizer $x_m$ exists, is unique, and resides on the interval $[-1,-1/2]$ for all $m\in\Bbb E$. Furthermore, we were able to establish a very simple expression for $\inf f_m$ as a function of this minimizer so that the problem of evaluating $\inf f_m$ is equivalent to finding $x_m$. For $m=2,4$ we may apply the standard equations for roots of linear and cubic equations to derive exact algebraic expressions for $x_m$. Furthermore, as $m\to\infty$ we find for $|x|<1$: $f_m(x)\to(1-x)^{-1}$, which is strictly increasing on $x\in(-1,-1/2]$. Bringing these observations together we have
\[
\begin{aligned}
x_2 &=-\frac{1}{2}\\
x_4 &=-\frac{1}{4} \left(1+\sqrt[3]{5/9} \left(\sqrt[3]{9+4 \sqrt{6}}-\sqrt[3]{4\sqrt{6}-9}\right)\right)\\
&\ \,\vdots\\
x_\infty &=-1.
\end{aligned}
\]

While a general algebraic solution for $x_m$ with $m\geq 6$ does not exist, methods for expressing exact solutions to higher-order polynomial roots have been thoroughly studied \cite{Umemura2007}. For example, the work of Hermite shows that $x_6$ can be solved exactly in terms of nonelementary functions \cite{hermite1858}. One way this is accomplished is by reducing the quintic equation $\partial_x f_6(x)=0$ to its Bring--Jerrard normal form and then using series reversion to express $x_6$ in terms of hypergeometric functions. Using this approach as a clue, Theorem \ref{thm:xm_closed_form} presents an exact and general solution for $x_m$ based on an adaptation of the method used for solving trinomial equations \cite{glasser_2000,Eagle_1939,Ritelli_2021}.

\begin{theorem}
\label{thm:xm_closed_form}
For all $m\in\Bbb E$
\[
x_m=\sum_{k=1}^m\frac{(-m)^{k-2}}{(1+m)^{\frac{mk+k}{m}-1}}\frac{\Gamma\left(\frac{mk+k}{m}-1\right)}{\Gamma\left(\frac{m+k}{m}\right)\Gamma(k)}%
\pFq{m+2}{m+1}{1,\{\frac{k}{m}+\frac{\ell-1}{m+1}\}_{\ell=0}^m}{\frac{m+k}{m},\{\frac{k+\ell}{m}\}_{\ell=0}^{m-1}}{1}.
\]
\end{theorem}

\begin{proof}
From Theorem \ref{thm:fn_has_a_unique_minimizer} we know $x_m$ satisfies
\[
mx_m^{m+1}-(m+1)x_m^m+1=0,\quad m\in\Bbb E.
\]
Performing the substitution $x_m\mapsto -\zeta^{-\frac{1}{m}}$ we obtain the transformed expression
\begin{equation}
\label{eq:xformed_root_eq}
\zeta=1+m+m\phi(\zeta),
\end{equation}
with $\phi(\zeta)=\zeta^{-\frac{1}{m}}$. By Lagrange's inversion theorem it follows for a function $F$ analytic in a neighborhood of the root of (\ref{eq:xformed_root_eq}) that
\[
F(\zeta)=F(1+m)+\sum_{n=1}^\infty\frac{m^n}{n!}\left[\partial_w^{n-1}F^\prime(w)\phi^n(w)\right]_{w=1+m}.
\]
Choosing $F(\zeta)=-\zeta^{-\frac{1}{m}}$ we subsequently obtain
\[
x_m=-(1+m)^{-\frac{1}{m}}+\sum_{n=1}^\infty\frac{m^{n-1}}{n!}\left[\partial_w^{n-1}w^{-\frac{m+n+1}{m}}\right]_{w=1+m},
\]
which upon further noting that $\partial_w^{n-1}w^{-s}=(-1)^{n-1}(s)_{n-1}w^{-s-n+1}$ yields after some algebraic manipulation
\begin{equation}
\label{eq:xm_Lagrange_series}
x_m=-\frac{(1+m)^{-\frac{1}{m}}}{m}\sum_{n=0}^\infty\frac{\Gamma\left(\frac{mn+n+1}{m}\right)}{\Gamma\left(\frac{m+n+1}{m}\right)}\frac{\bigl(-m(1+m)^{-\frac{m+1}{m}}\bigr)^n}{n!}.
\end{equation}
To evaluate the series (\ref{eq:xm_Lagrange_series}) we write $x_m=\sum_{n=0}^\infty c_n=\sum_{k=1}^m\sum_{n=0}^\infty c_{mn+k-1}$; resulting in $m$ new series containing Pochhammer symbols of the form $(\cdot)_{(m+1)n}$ and $(\cdot)_{mn}$. Then using the identity \cite[Eq.~2.13]{JIMS252}
\[
(\alpha)_{rn}=r^{rn}\prod_{j=0}^{r-1}\left(\frac{\alpha+j}{r}\right)_n,\quad r\in\Bbb N
\]
we arrive at
\[
x_m=\sum_{k=1}^m\frac{(-m)^{k-2}}{(1+m)^{\frac{mk+k}{m}-1}}\frac{\Gamma\left(\frac{mk+k}{m}-1\right)}{\Gamma\left(\frac{m+k}{m}\right)\Gamma(k)}%
\sum_{n=0}^\infty\frac{(1)_n\prod_{\ell=0}^m\bigl(\frac{k}{m}+\frac{\ell-1}{m+1}\bigr)_n}{\left(\frac{m+k}{m}\right)_n\prod_{\ell=0}^{m-1}\left(\frac{k+\ell}{m}\right)_n}\frac{1}{n!},
\]
which is the desired result.
\end{proof}

To demonstrate the validity of the closed-form for $x_m$ given by Theorem \ref{thm:xm_closed_form} we substitute $m=2$ and find
\[
x_2=\frac{1}{9}\pFq{3}{2}{1,\frac{2}{3},\frac{4}{3}}{2,\frac{3}{2}}{1}-\frac{1}{\sqrt 3}\pFq{2}{1}{\frac{1}{6},\frac{5}{6}}{\frac{3}{2}}{1}.
\]
The $_3F_2(\cdot)$ term is reduced to $_2F_1(\cdot)$ by way of \cite[Eq.~07.27.03.0120.01]{wolfram_functions}
\[
\pFq{3}{2}{1,\beta,\gamma}{2,\epsilon}{z}=\frac{\epsilon-1}{(\beta-1)(\gamma-1)z}\left(\pFq{2}{1}{\beta-1,\gamma-1}{\epsilon-1}{z}-1\right).
\]
Gauss's hypergeometric summation theorem
\[
\pFq{2}{1}{\alpha,\beta}{\gamma}{1}=\frac{\Gamma(\gamma)\Gamma(\gamma-\alpha-\beta)}{\Gamma(\gamma-\alpha)\Gamma(\gamma-\beta)},\quad\Re(\gamma-\alpha-\beta)>0
\]
then permits us to write the remaining $_2F_1(\cdot)$ terms as ratios of gamma functions. After some simplification we find
\[
x_2=-\frac{1}{2},
\]
which is the exact value of $x_2$.

For $m\geq 4$, reducing the closed-form for $x_m$ to more elementary functions in this manner becomes very cumbersome if not impossible. Without the ability to reduce the hypergeometric functions present in $x_m$, this expression also becomes difficult to implement numerically, especially as $m$ becomes large. To obtain approximations we could turn to the series given by (\ref{eq:xm_Lagrange_series}); however, the slow convergence of this series renders it impractical. For example, substituting $m=2$ and adding up the first one-hundred terms of (\ref{eq:xm_Lagrange_series}) we obtain $x_2\approx -0.499885$, which corresponds to an absolute relative error of $2.3\times 10^{-4}$. Given that numerical root finding methods can achieve more accurate approximations in just a few iterations we find this means of approximation to be less than satisfactory.


\section{Perturbation series expansion of $x_m$}
\label{sec:xn_pert_series_approx}

In the previous section we were able to find an exact expression for $x_m$ but this expression was not useful for the purpose of computing numerical approximations. Here, we apply the methods of perturbation theory to obtain a faster converging series expansion for this purpose.

We begin by recalling from Theorem \ref{thm:fn_has_a_unique_minimizer} that $x_m$ satisfies
\[
g_m(x_m)=0,\quad\text{with}\quad g_m(x)=x^m\left(1-x+\tfrac{1}{m}\right)-\tfrac{1}{m}
\]
and $x_m\to -1$ as $m\to\infty$. The fact that $x_m+1$ vanishes as $m$ becomes large suggests we instead study the perturbed problem
\[
g_{m,\epsilon}(x_{m,\epsilon})=0,\quad\text{with}\quad g_{m,\epsilon}(x)=x^m\left(2-(1+x)\epsilon+\tfrac{1}{m}\right)-\tfrac{1}{m},
\]
where
\begin{equation}
\label{eq:pert_series_gen_form}
x_{m,\epsilon}=\sum_{k=0}^\infty a_k\epsilon^k.
\end{equation}
Upon inspection we observe $g_{m,1}(x)=g_m(x)$ and so it follows that $x_m$ can be recovered by evaluating the perturbation series (\ref{eq:pert_series_gen_form}) at $\epsilon=1$. To determine the coefficients $a_k$ we first consider the well-known result for integer powers of series to express powers of $x_{m,\epsilon}$ as
\[
x_{m,\epsilon}^p=\sum_{k=0}^\infty c_{k,p}\epsilon^k,\quad p\in\Bbb N
\]
with
\[
\begin{aligned}
c_{0,p} &=a_0^p\\
c_{k,p} &=\frac{1}{a_0k}\sum_{\ell=1}^k((p+1)\ell-k)a_\ell c_{k-\ell,p}.
\end{aligned}
\]
Using Fa\'{a} di Bruno's formula we may also obtain a closed-form for the coefficients $c_{k,p}$ as
\[
c_{k,p}=\frac{1}{k!}\sum_{\ell=1}^k(p)^{(\ell)}a_0^{p-l}B_{k,\ell}(1!\, a_1,\dots,(k-\ell+1)!\,a_{k-\ell+1}),
\]
where $B_{n,k}(x_1,\dots,x_{n-k+1})$ is the partial Bell-polynomial. Using these results we substitute $x_{m,\epsilon}$ into $g_{m,\epsilon}$ and collect terms by powers of $\epsilon$ yielding
\[
g_{m,\epsilon}(x_{m,\epsilon})=\left(2+\tfrac{1}{m}\right)a_0^m-\tfrac{1}{m}+\sum_{k=1}^\infty\left[\left(2+\tfrac{1}{m}\right)c_{k,m}-c_{k-1,m}-c_{k-1,m+1}\right]\epsilon^k.
\]
Since $g_{m,\epsilon}(x_{m,\epsilon})=0$ we equate the coefficients of $\epsilon^k$ with zero to yield an infinite system of equations that recover the coefficients $a_k$. Setting the constant term equal to zero gives $a_0^m=(1+2m)^{-1}$. Knowing that $x_m\in[-1,-1/2]$ and $m\in\Bbb E$ we take the negative solution to this equation and set the higher-order coefficients of $\epsilon^k$ to zero yielding
\begin{equation}
\label{eq:recurrence_relation}
a_0=-(1+2m)^{-\frac{1}{m}},\quad (1+2m)c_{k,m}-m(c_{k-1,m}+c_{k-1,m+1})=0.
\end{equation}
Evaluating the first several coefficients we are able to conjecture a closed-form for $a_k$, which leads to the following result.

\begin{theorem}
\label{thm:xm_closed_form_2}
For all $m\in\Bbb E$
\[
x_m=\sum_{k=0}^\infty a_k,
\]
where
\[
\begin{aligned}
a_0 &=-(1+2m)^{-\frac{1}{m}}\\
a_k &=\sum_{\ell=0}^k\frac{(\ell+m+1)_{k-1,m}}{\ell! (k-\ell)!}a_0^{mk+\ell+1}.
\end{aligned}
\]
\end{theorem}

\begin{proof}
We begin by considering the closed-form for $x_m$ claimed in the statement of Theorem \ref{thm:xm_closed_form}, which consists of a sum of $m$ hypergeometric functions. Denoting $\{a_{j,k}\}_{j=1}^{m+2}$ as the top parameters and $\{b_{j,k}\}_{j=1}^{m+1}$ as the bottom parameters of the hypergeometric function in the $k$th term we find $\gamma_k=(b_1+\cdots+b_{m+1})-(a_1+\dots+a_{m+2})=\frac{1}{2}$ for all $k=1,\dots,m$. Since $\gamma_k>0$, each of the $m$-terms of $x_m$ can be written as an absolutely convergent series; hence, the entire expression representing $x_m$ must also be absolutely convergent. Now using the conjectured closed-form for $a_k$ we write
\[
x_m=\sum_{k=0}^\infty\sum_{\ell=0}^k\frac{m^{k-1}}{\ell! (k-\ell)!}\frac{\Gamma\left(k+\frac{\ell+1}{m}\right)}{\Gamma\left(1+\frac{\ell+1}{m}\right)}a_0^{mk+\ell+1}.
\]
If this expression is equal to that given in the statement of Theorem \ref{thm:xm_closed_form}, then it is also absolutely convergent and permits rearrangement of its terms. Interchanging the order of summation we find after some simplification
\[
x_m=\frac{a_0}{m}\sum_{\ell=0}^\infty\frac{\Gamma\left(\frac{m\ell+\ell+1}{m}\right)}{\Gamma\left(1+\frac{\ell+1}{m}\right)}\frac{(ma_0^{m+1})^\ell}{\ell !}\sum_{k=0}^\infty\left(\tfrac{m\ell+\ell+1}{m}\right)_k\frac{(ma_0^m)^k}{k!}.
\]
The interior sum over $k$ can now be evaluated in terms of ${_1F}_0(\alpha;-;z)=(1-z)^{-\alpha}$. Reintroducing $a_0$ yields
\[
x_m=-\frac{(1+m)^{-\frac{1}{m}}}{m}\sum_{\ell=0}^\infty\frac{\Gamma\left(\frac{m\ell+\ell+1}{m}\right)}{\Gamma\left(\frac{m+\ell+1}{m}\right)}\frac{\bigl(-m(1+m)^{-\frac{m+1}{m}}\bigr)^\ell}{\ell !},
\]
which is the series expansion for $x_m$ given in (\ref{eq:xm_Lagrange_series}). By the uniqueness of Taylor series it follows that the conjectured form for $a_k$ must be correct.
\end{proof}

So does the perturbation series for $x_m$ converge faster than that given by (\ref{eq:xm_Lagrange_series})? Substituting $m=2$ and adding the first one-hundred terms we find for the absolute relative error $5.6\times 10^{-64}$, which is a significant improvement on the absolute relative error of $2.3\times 10^{-4}$ obtained from the first one hundred terms of (\ref{eq:xm_Lagrange_series}).

We conclude this section with an important property of approximations for $x_m$ obtained via the perturbation series of Theorem \ref{thm:xm_closed_form_2}.

\begin{corollary}
\label{cor:xmn_is_asymptotically_exact}
If $\tilde x_{m,n}=\sum_{k=0}^na_k$, then $x_m\sim\tilde x_{m,n}$ as $m\to\infty$.
\end{corollary}

\begin{proof}
Using the expression for $a_k$ given in Theorem \ref{thm:xm_closed_form_2} we have $\lim_{m\to\infty}a_0=-1$ and $\lim_{m\to\infty}a_k=0$ for all $k\geq1$; thus, $\lim_{m\to\infty}\tilde x_{m,n}=-1$ for all $n\in\Bbb N_0$. Since $\lim_{m\to\infty}x_m=-1$ the result follows.
\end{proof}


\section{Numerical results}
\label{sec:numerical analysis}

From Corollary \ref{cor:xmn_is_asymptotically_exact} we know $x_m\sim\tilde x_{m,n}$ as $m\to\infty$ and so we expect the number $n$ needed to guarantee $|x_m-\tilde x_{m,n}|<\epsilon$ should decrease as $m$ increases. Since we have closed-forms for $x_2$ and $x_4$, which can be computed to arbitrary precision, our first task will be to study the convergence of $\tilde x_{m,n}\to x_m$ as a function of $n$ for $m=2,4$. Given that we expect less terms will be needed for larger values of $m$, the results of this exercise should give us a worst case scenario for how large $n$ must be to obtain the desired accuracy in our approximation.

Using {\sc Mathematica} software, we evaluated $\tilde x_{m,n}$ for $m=2,4$ and $n=0,1,\dots, 100$. To compare the approximation to the exact values we used the absolute relative error
\[
R_m(n)=\left\lvert\frac{\tilde x_{m,n}}{x_m}-1\right\rvert,
\]
the results of which are plotted in Figure \ref{fig:abs_rel_error}. From the figure we see the error decreases exponentially with $n$ and that $R_4(n)<R_2(n)$ for each value of $n$. Working with the data for $R_4(n)$ we further determined
\[
R_4(n)<5\times 10^{-(2+0.759n)},
\]
which suggests setting $n$ equal to
\[
n^\ast=\max\left\{0,\left\lceil\tfrac{q-2}{0.759}\right\rceil\right\}
\]
is sufficient to guarantee $\tilde x_{m,n^\ast}$ agrees with $x_m$ to at least $q$ significant digits for all $m\geq4$.

To test this hypothesis we first note that $x_m\in(-1,-1/2]$ for all finite $m\in\Bbb E$. Since the leading exponent in the decimal expansion of $x_m$ is always negative one it follows that $\tilde x_{m,n}$ has $p$ significant digits of $x_m$ if $|x_m-\tilde x_{m,n}|\leq 5\times 10^{-(p+1)}$. Furthermore, we know $x_m$ is the unique real root of $g_m(x)=x^m(1-x+\frac{1}{m})-\frac{1}{m}$ with $g_m(x_m-\epsilon)$ and $g_m(x_m+\epsilon)$ differing in sign; hence a lower bound on the number of significant digits obtained by $\tilde x_{m,n}$ is found by determining the largest nonnegative integer $p$ such that
$$
g_m(\tilde x_{m,n}-5\times 10^{-(p+1)})g_m(\tilde x_{m,n}+5\times 10^{-(p+1)})\leq 0.
$$

For the sake of example we chose $q=10$ for the number of desired significant digits resulting in $n^\ast=11$. Using the above mentioned procedure, the value $p$ was computed for $m=4,6,\dots,100$ with the results presented in Figure \ref{fig:NsigFigs_q_10}. From the figure we observe $p>q$ for each value of $m$ as is expected. Finally, Table \ref{tbl:fmin_data} presents numerical values for $x_m$ and $f_m(x_m)$ computed using $\tilde x_{m,n}$ and the formula in Lemma \ref{lem:fmin_explicit_form}.

\begin{figure}[htb]
\centering
\includegraphics[scale=1]{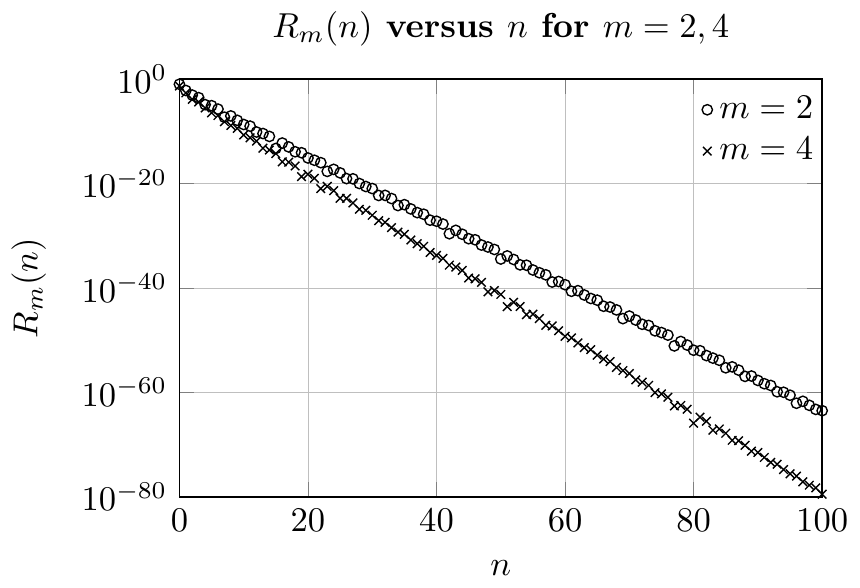}
\caption{Absolute relative error incurred from the approximation $\tilde x_{m,n}$ versus $n$ for $m=2,4$. Plot produced with matlab2tikz \cite{schlomer_2021}.}
\label{fig:abs_rel_error}
\end{figure}

\begin{figure}[htb]
\centering
\includegraphics[scale=1]{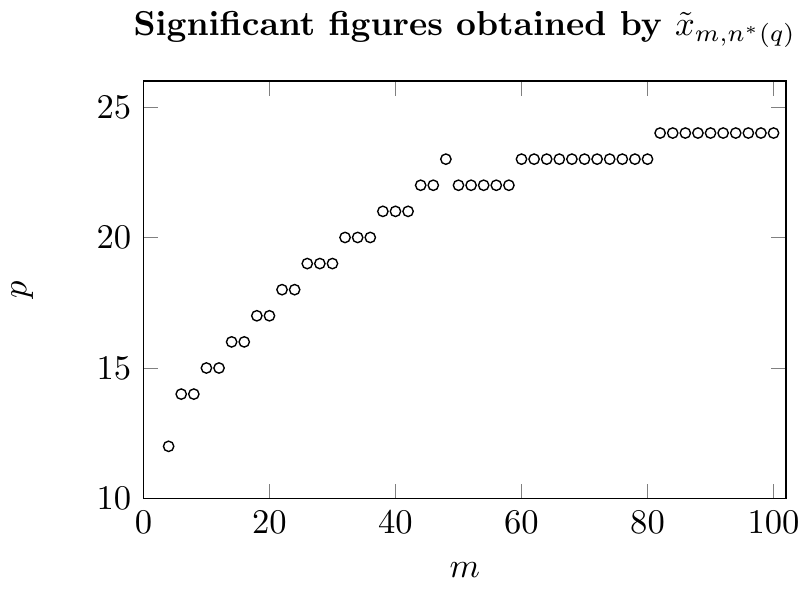}
\caption{Number of significant figures obtained by the approximation $\tilde x_{m,n^\ast}(q)$ for $q=10$ versus $m$.}
\label{fig:NsigFigs_q_10}
\end{figure}

\begin{table}[htb]
\centering
\caption{Numerical values for $x_m$ and $f_m(x_m)$.}
\vspace*{3pt}
\begin{tabular}{ccc}\toprule
$m$ & $x_m$ & $f_m(x_m)$\\\midrule
 2 & -0.5000000000 & 0.7500000000 \\
 4 & -0.6058295862 & 0.6735532235 \\
 6 & -0.6703320476 & 0.6350938940 \\
 8 & -0.7145377272 & 0.6115666906 \\
 10 & -0.7470540749 & 0.5955429324 \\
 12 & -0.7721416355 & 0.5838576922 \\
 14 & -0.7921778546 & 0.5749221276 \\
 16 & -0.8086048979 & 0.5678463037 \\
 18 & -0.8223534102 & 0.5620909079 \\
 20 & -0.8340533676 & 0.5573090540 \\
 22 & -0.8441478047 & 0.5532669587 \\
 24 & -0.8529581644 & 0.5498010211 \\
 26 & -0.8607238146 & 0.5467931483 \\
 28 & -0.8676269763 & 0.5441558518 \\
 30 & -0.8738090154 & 0.5418228660 \\
 32 & -0.8793814184 & 0.5397430347 \\
 34 & -0.8844333818 & 0.5378762052 \\
 36 & -0.8890371830 & 0.5361903986 \\
 38 & -0.8932520563 & 0.5346598151 \\
 40 & -0.8971270425 & 0.5332633990 \\
 42 & -0.9007031162 & 0.5319837878 \\
 44 & -0.9040147981 & 0.5308065300 \\
 46 & -0.9070913919 & 0.5297194951 \\
 48 & -0.9099579456 & 0.5287124219 \\
 50 & -0.9126360054 & 0.5277765690 \\
 52 & -0.9151442141 & 0.5269044410 \\
 54 & -0.9174987898 & 0.5260895727 \\
 56 & -0.9197139122 & 0.5253263565 \\
 58 & -0.9218020367 & 0.5246099035 \\
 60 & -0.9237741513 & 0.5239359311 \\
 62 & -0.9256399895 & 0.5233006711 \\
 64 & -0.9274082062 & 0.5227007942 \\
 66 & -0.9290865244 & 0.5221333471 \\
 68 & -0.9306818591 & 0.5215957008 \\
 70 & -0.9322004214 & 0.5210855067 \\
 72 & -0.9336478067 & 0.5206006599 \\
 74 & -0.9350290699 & 0.5201392683 \\
 76 & -0.9363487901 & 0.5196996259 \\\bottomrule
\end{tabular}\ \ \ 
\begin{tabular}{ccc}\toprule
$m$ & $x_m$ & $f_m(x_m)$\\\midrule
 78 & -0.9376111258 & 0.5192801905 \\
 80 & -0.9388198625 & 0.5188795643 \\
 82 & -0.9399784542 & 0.5184964771 \\
 84 & -0.9410900592 & 0.5181297723 \\
 86 & -0.9421575717 & 0.5177783938 \\
 88 & -0.9431836485 & 0.5174413759 \\
 90 & -0.9441707340 & 0.5171178332 \\
 92 & -0.9451210804 & 0.5168069528 \\
 94 & -0.9460367670 & 0.5165079864 \\
 96 & -0.9469197164 & 0.5162202447 \\
 98 & -0.9477717091 & 0.5159430910 \\
 100 & -0.9485943966 & 0.5156759367 \\
 102 & -0.9493893132 & 0.5154182363 \\
 104 & -0.9501578860 & 0.5151694840 \\
 106 & -0.9509014444 & 0.5149292100 \\
 108 & -0.9516212282 & 0.5146969770 \\
 110 & -0.9523183955 & 0.5144723780 \\
 112 & -0.9529940289 & 0.5142550329 \\
 114 & -0.9536491420 & 0.5140445872 \\
 116 & -0.9542846846 & 0.5138407092 \\
 118 & -0.9549015479 & 0.5136430885 \\
 120 & -0.9555005690 & 0.5134514340 \\
 122 & -0.9560825347 & 0.5132654729 \\
 124 & -0.9566481855 & 0.5130849489 \\
 126 & -0.9571982191 & 0.5129096209 \\
 128 & -0.9577332933 & 0.5127392623 \\
 130 & -0.9582540286 & 0.5125736594 \\
 132 & -0.9587610115 & 0.5124126108 \\
 134 & -0.9592547961 & 0.5122559267 \\
 136 & -0.9597359069 & 0.5121034274 \\
 138 & -0.9602048403 & 0.5119549436 \\
 140 & -0.9606620669 & 0.5118103146 \\
 142 & -0.9611080328 & 0.5116693885 \\
 144 & -0.9615431615 & 0.5115320215 \\
 146 & -0.9619678551 & 0.5113980772 \\
 148 & -0.9623824957 & 0.5112674259 \\
 150 & -0.9627874469 & 0.5111399447 \\
 $\infty$ &-1.0000000000 &0.5000000000\\\bottomrule
\end{tabular}
\label{tbl:fmin_data}
\end{table}


\section{Conclusions}
\label{sec:conclusions}

In this note, we were able to establish many useful facts about the polynomial $f_m(x)=1+x+\cdots+x^m$ and its minimum value on the real line. In particular, we were able to show $\arginf f_m(x)\in[-1,-1/2]$ and $\inf f_m(x)\in[1/2,3/4]$ for all $m\in\Bbb E$ as well as provide a very simple formula for the minimum as a function of the minimizer $x_m$. Lagrange inversion and perturbation theory were applied to derive two different series expansions for $x_m$, which lead to a closed-form in terms of hypergeometric functions. Furthermore, numerical studies were conducted which gave a rule of thumb for how large $n$ must be to achieve a desired accuracy in approximating $x_m$ with $\tilde x_{m,n}$.


\bibliographystyle{plain}
\bibliography{mybibfile}

\begin{thebibliography}{10}

\bibitem{MR1577598}
N.~H. Abel.
\newblock Beweis der {U}nm\"{o}glichkeit, algebraische {G}leichungen von
  h\"{o}heren {G}raden als dem vierten allgemein aufzul\"{o}sen.
\newblock {\em J. Reine Angew. Math.}, 1:65--84, 1826.

\bibitem{diaz_2007}
Rafael D\'{\i}az and Eddy Pariguan.
\newblock On hypergeometric functions and {P}ochhammer $k$-symbol.
\newblock {\em Divulgaciones Matem\'{a}ticas}, 15(2), 2007.

\bibitem{Eagle_1939}
Albert Eagle.
\newblock Series for all the roots of a trinomial equation.
\newblock {\em The American Mathematical Monthly}, 46(7):422--425, 1939.

\bibitem{glasser_2000}
M.L. Glasser.
\newblock Hypergeometric functions and the trinomial equation.
\newblock {\em Journal of Computational and Applied Mathematics},
  118(1):169--173, 2000.

\bibitem{hermite1858}
Charles Hermite.
\newblock {S}ur la r{\'e}solution de l’{\'e}quation du cinquieme degr{\'e}.
\newblock {\em Comptes rendus de l’Acad{\'e}mie des Sciences},
  46(1858):508--515, 1858.

\bibitem{wolfram_functions}
Wolfram~Research Inc.
\newblock The {W}olfram functions site.
\newblock Visited on 2021-04-28.

\bibitem{JIMS252}
Shahid Mubeen and Abdur Rehman.
\newblock A note on $k$-gamma function and pochhammer $k$-symbol.
\newblock {\em Journal of Informatics and Mathematical Sciences}, 6(2), 2014.

\bibitem{Ritelli_2021}
Daniele Ritelli and Giulia Spaletta.
\newblock Trinomial equation: the hypergeometric way.
\newblock {\em Open Journal of Mathematical Sciences}, 5(1):236--247, 2021.

\bibitem{schlomer_2021}
Nico Schl{\"{o}}mer.
\newblock matlab2tikz: {A} script to convert {MATLAB}/{O}ctave into {T}ik{Z}
  figures for easy and consistent inclusion into {\LaTeX}.
\newblock GitHub.
\newblock URL: https://github.com/matlab2tikz/matlab2tikz (retrieved May 8,
  2021).

\bibitem{Umemura2007}
Hiroshi Umemura.
\newblock {\em Resolution of algebraic equations by theta constants}, pages
  261--270.
\newblock Birkh{\"a}user Boston, Boston, MA, 2007.

\bibitem{wang_4060608}
Yiwei Wang.
\newblock The minimum of $f(x)=1+x+\cdots+x^{2n}$.
\newblock Mathematics Stack Exchange.
\newblock URL: https://math.stackexchange.com/q/4060608 (version: 2021-04-30).

\end{thebibliography}

\end{document}